\documentclass[12pt]{amsart}
\usepackage{url}
\usepackage[all]{xy}
\usepackage{array, tabularx, longtable}
\usepackage{amsthm, amsmath, amscd, amssymb,amsfonts}
\usepackage[mathscr]{eucal}
\usepackage[all]{xypic}
\newtheorem{thm}{Theorem}[section]
\newtheorem{theorem}[thm]{Theorem}

\newtheorem{lemma}[thm]{Lemma}
\newtheorem{claim}[thm]{Claim}
\newtheorem{proposition}[thm]{Proposition}
\newtheorem{example}[thm]{Example}

\numberwithin{equation}{section}
\usepackage[top=3.0cm, bottom=3.0cm, left=3.0cm, right=3.0cm]{geometry}
\begin{document}
\baselineskip=16pt

\title[Parametrization simple irreducible plane curve singularities]{Parametrization simple irreducible plane curve singularities in arbitrary characteristic}      
\author{Nguyen Hong Duc}
\address{$^{\dag}$Basque Center for Applied Mathematics, \newline \indent Alameda de Mazarredo 14, 48009 Bilbao, Bizkaia, Spain.} 
\email{hnguyen@bcamath.org}

\address{$^{\dag}$TIMAS, Thang Long University, \newline \indent Nghiem Xuan Yem, Hanoi, Vietnam.} 
\email{nhduc82@gmail.com}
\thanks{The author's research is supported by Juan de la Cierva Incorporación IJCI-2016-29891, the ERCEA Consolidator Grant 615655 NMST and the National Foundation for Science and Technology Development (NAFOSTED), Grant number 101.04-2017.12, Vietnam.}
\date{\today}                  
\maketitle
\begin{abstract}
We study the classification of plane curve singularities in arbitrary characteristic. We first give a bound for the determinacy of a plane curve singularity with respect to pararametrization equivalence in terms of its conductor. Then we classify parametrization simple plane curve singularities which are irreducible by giving a concrete list of normal forms of equations and parametrizations. In characteristic zero, the classification of parametrization simple irreducible plane curve singularities was achieved by Bruce and Gaffney.
\end{abstract}
\section{Introduction}
We classify irreducible plane curve singularities $f\in K[[x,y]]$ which are simple with respect to parametrization equivalence, where $K$ is an algebraically closed field of arbitrary characteristic. That is, the irreducible plane singularities whose parametrizations have modality 0 up to the change of coordinates in the source and target spaces (or, left-right equivalence, see Section \ref{sec2.1}). The notion of modality was introduced by Arnol'd in the seventies into the singularity theory for real and complex singularities. He classified simple, unimodal and bimodal hypersurface singularities with respect to right equivalence, i.e. the hypersurface singularities of right modality 0,1,2 respectively \cite{Arn72},\cite{Arn73},\cite{Arn76}. The classifications of contact simple and unimodal complete intersection singularities were done by Giusti \cite{Giu77} and Wall \cite{Wal83}. Classification of contact simple space curve singularities was  obtained by Giusti \cite{Giu77} and Fr\"uhbis-Kr\"uger \cite{FK99}. In positive characteristic, the right simple, unimodal and bimodal hypersurface singularities were recently classified by Greuel and the author in \cite{GN16} and \cite{Ng17}. The classification of contact simple hypersurface singularities were achieved by Greuel-Kr\"oning \cite{GK90}, while classifications of contact unimodal and bimodal singularities are still unknown. 

Curve singularities can be also described by parametrisations. Two plane curve singularities are contact equivalent if and only if their parametrizations are left-right equivalent. The first results on classification of simple curve singularities with respect to parametrization equivalence were obtained by Bruce and Gaffney, for complex irreducible plane curve singularities in $\mathbb{C}\{x,y\}$ \cite{BrG82}. The classifications were extended to irreducible space curves by Gibson and Hobbs \cite{GH93}, irreducible curves of any embedding dimension by Arnold \cite{Arn99}, and reducible curves by Kolgushkin and Sadykov \cite{KS01}.

In this paper, we generalize the result of Bruce and Gaffney to the singularities in arbitrary characteristic (Theorem \ref{thm31}). We give lists of normal forms of equations and parametrizations of parametrization simple plane curve singularities which are irreducible (Tables \ref{table1},\ref{table2},\ref{table3} in Section \ref{sec3}). We first study in Section \ref{sec2} the problem of determinacy with respect to parametrization equivalence. The theory of determinacy was systematically studied by Mather in \cite{Mat68}, where he defined the equivalence relations $\mathscr R, \mathscr C, \mathscr K, \mathscr L$ and $\mathscr A$ and obtained necessary and sufficient conditions for finite determinacy with respect to them. He also gave estimates for the corresponding determinacy. Lower estimates were provided later by Gaffney, Bruce, du Plessis and Wall. The problem of determinacy in positive characteristic with respect to $\mathscr R, \mathscr K$ was treated by Boubakri, Greuel and Markwig in \cite{BGM12} and recently by Greuel and Pham \cite{GP16},\cite{GP17}. We show that reduced plane curve singularities are finitely determined with respect to parametrization equivalence. Moreover, we give a lower bound for parametrization determinacy of a plane curve singularity in terms of its conductor (Theorem \ref{thm21}).

\subsection*{Acknowledgement} 
A part of this article was done in my thesis under the supervision of Professor Gert-Martin Greuel at the Technische Universit\"at Kaiserslautern. I am grateful to him for many valuable suggestions. 
\section{Parametrization determinacy}\label{sec2}
\subsection{Parametrization equivalence}\label{sec2.1}
For a {\em plane curve singularity} $f$, i.e. an element in the maximal ideal $\mathfrak m$ in $K[[x, y]]$, there is a unique (up to multiplication with units) decomposition 
$f=f_1^{\rho_1}\cdot\ldots\cdot f_r^{\rho_r},$
with $f_i\in \mathfrak m$ irreducible in $K[[x,y]]$. We assume, in this note, that $f$ is {\em reduced}, i.e. $\rho_i=1$ for all $i=1,\ldots,r$. The integral closure of $R:=R_f:=K[[x,y]]/\langle f\rangle$ (in the total quotient ring Quot($R$)) is isomorphic to $ \bar{R}:=  \bigoplus_{i=1}^r K[[t]]$ (see \cite{Cam80}, \cite{GLS06}). A composition $K[[x, y]] \twoheadrightarrow R  \hookrightarrow \bar{R} =\bigoplus_{i=1}^r K[[t]]$ of the natural projection $K[[x,y]]\twoheadrightarrow R$ and a normalization $R  \hookrightarrow \bar{R}$, is called a {\em parametrization} of $f$. It is an element in the space $J:=\mathrm{Hom}_K(K[[x,y]],\bar{R})$ of morphisms of local $K$-algebras. Any element of $\psi\in J$ can be identified with the image of $\psi(x),\psi(y)$ in $\bar{R}$. Hence, it is often written as a tuple of $r$ pairs $(x_i(t),y_i(t))$.  . 

Two morphisms of $K$-algebras $\psi,\psi'\colon K[[x,y]]\rightarrow  \bar{R}=\bigoplus_{i=1}^r K[[t]]$ are called {\em left-right equivalent} (or, $\mathcal{A}$-equivalent), $\psi \sim_{\mathcal{A}} \psi'$, if there exist an automorphism $\phi$ of $ \bar{R}$ and an automorphism $\Phi\in Aut_K(K[[x,y]])$ such that $\Phi\circ\psi=\phi\circ\psi'$. By an automorphism of $\bar{R}$ we mean a tuple of automorphisms of $K[[t]]$. Two plane curves $f,g \in K[[x,y]]$ are called {\em parametrization equivalent}, denoted by $f\sim_{p} g$, if there exist a parametrization $\psi$ of $f$ and a parametrization $\psi'$ of $g$ such that $\psi \sim_{\mathcal{A}} \psi'$. It was known that, $f\sim_{p} g$ if and only if $f\sim_{c} g$ (\cite[Prop. 1.2.10]{Ng13}, see also \cite[Lemma 2.2]{BrG82} for $f$ irreducible). 
\subsection{Parametrization determinacy}
For each ${\bf k}=(k_1,\ldots,k_r)\in\Bbb Z_{\geq 0}^r$, the ${\bf k}$-jet of $\psi$ is defined to be the composition $j^{\bf k}\psi\colon K[[x, y]] \overset{\psi}{\to} \bigoplus_{i=1}^r K[[t]]\to \bigoplus_{i=1}^r K[[t]]/(t^{k_i+1}).$ We call $\psi$ {\em parametrization ${\bf k}$-determined} if it is parametrization equivalent to every $\psi'$ whose ${\bf k}$-jet coincides with that of $\psi$. We say that $f$ is {\em parametrization finitely determined} if one (and therefore all) of its parametrizations is parametrization ${\bf k}$-determined for some ${\bf k}=(k_1,\ldots,k_r)\in\Bbb Z_{\geq 0}^r$. A minimum ${\bf k}$ with this property is called a parametrization determinacy of $f$ (or $\psi$). We show, in the present note, that $f$ is ${\bf d}$-parametrization determined, where ${\bf d}$ is concretely given by the conductor of $f$. 

Let $\mathcal C:=(R:\bar{R}):=\{u\in R\ |\ u\bar{R}\subset R\}$ be the {\em conductor ideal} of $\bar{R}$ in $R$ (cf. \cite{ZS60}). Then $\mathcal C$ is an ideal of both $R$ and $\bar{R}$. So one has $\mathcal C=(t^{c_1})\times \cdots \times(t^{c_r})$ for some $c_1,\ldots,c_r\in \Bbb Z_{\geq 0}$. We call ${\bf c}:={\bf c}(f):=(c_1,\ldots,c_r)\in \Bbb Z_{\geq 0}^r$ the {\em conductor (exponent)} of $f$. 
The conductor ${\bf c}=(c_1,\ldots,c_r)$ of $f$ is related to the ones of its branches and other invariants by the following beautiful formulas
\begin{equation}\label{eq21}
c_i=c(f_i)+\sum_{j\neq i} i(f_i,f_j)
\end{equation}
and
\begin{equation}\label{eq22}
|{\bf c}|:=c_1+\ldots+c_r=2\delta,
\end{equation}
where $\delta$ is the delta invariant of $f$, defined as $\delta:=\dim_K \bar{R}/R$. \\Here for $g,h\in K[[x,y]]$, $i(g,h)$ denotes the {\em intersection multiplicity} of $g,h$ defined by $i(g,h):=\dim_K K[[x,y]]/(g,h).$
Note that, if $h$ is irreducible and $\psi$ is a parametrization of $h$, then $i(g,h)=\text{ord }\psi(g)$. Furthermore, the intersection multiplicity is additive, i.e. if $h=h_1\cdot\ldots\cdot h_r$, then $i(g,h)=i(g,h_1)+\ldots+i(g,h_r).$ 

\begin{theorem}\label{thm21}
Let $f\in\mathfrak{m}\subset K[[x,y]]$ be reduced, $r$ the number of the irreducible components, ${\bf c}\in \Bbb Z_{\geq 0}^r$ its conductor, and let
\begin{equation*}
\Bbb Z_{\geq 0}^r \ni {\bf d}:=
\begin{cases}
1 &\text{ if } \mathrm{mt}(f)=1\\
{\bf c}+1 &\text{ if } \mathrm{mt}(f)=2 \text{ and } r=1\\
{\bf c}&\text{ if } \mathrm{mt}(f)=2 \text{ and } r=2\\
{\bf c}-{\bf 1} &\text{ if } \mathrm{mt}(f)>2.
\end{cases}
\end{equation*}
Then $f$ is parametrization ${\bf d}$-determined. In particular, $f$ is always parametrization $({\bf c}+{\bf 1})$-determined.
\end{theorem}
The multiplicity of $f$, $\mathrm{mt}(f)$, is defined to be the maximal of integers $k$ for which $\langle f\rangle\subset \mathfrak{m}^k$. For the proof of the theorem we need the two following lemmas, which give several relations between the conductor (${\bf c}$) and the maximal contact multiplicity ($\bar{\beta_1}$) of a reduced power series $f$ in some concrete cases. Recall that the {\em maximal contact multiplicity} of $f$ is defined by
$$\bar{\beta_1}(f):=\sup \{\min_{i=1,\ldots,r} i(f_i,\gamma)|\gamma \text{ regular}\},$$
where $f_1,\ldots,f_r$ are the irreducible components of $f$. We omit proofs of the lemmas here and refer to \cite{Ng13}, Lemma 2.5.4 and 2.5.5, since they are elementary.  
\begin{lemma}\label{lm21}
Let $f=f_1\cdot f_2\in K[[x,y]]$ be reduced such that $f_1,f_2$ are regular. Then 
$$\bar{\beta_1}(f)= i(f_1,f_2).$$
\end{lemma}
\begin{lemma}\label{lm22}
Let $f\in K[[x,y]]$ be irreducible. 
\begin{itemize}
\item[(i)] If $\mathrm{mt}(f)=2$, then $c(f)=\bar{\beta_1}(f)-1$.
\item[(ii)] If $\mathrm{mt}(f)>2$, then $c(f)>\bar{\beta_1}(f)$.
\end{itemize}
\end{lemma}
\begin{proof}[Proof of Theorem \ref{thm21}]
Note that ${\bf d}+{\bf 1}\geq {\bf c}$, i.e. $d_i+1\geq c_i$ for all $i=1,\ldots, r$. Let $\psi=(\psi_1,\ldots,\psi_r)\colon K[[x,y]]\to \bar{R}$ be a parametrization of $f$ and let $\psi'\colon K[[x,y]]\to \bar{R}$ such that $j^{\bf d}(\psi)=j^{\bf d}(\psi')$. It suffices to show that $\psi \sim_{\mathcal A} \psi'$. 

Indeed, we have
$$\psi(x)-\psi'(x)\in t^{{\bf d}+{\bf 1}}\bar{R}\subset R\text{ and }\psi(y)-\psi'(y)\in t^{{\bf d}+{\bf 1}}\bar{R}\subset R.$$
Thus there exist $g_1,g_2\in K[[x,y]]$ such that 
$$\psi(g_1)=\psi(x)-\psi'(x)\in t^{{\bf d}+{\bf 1}}\bar{R}\text{ and }\psi(g_2)=\psi(y)-\psi'(y)\in t^{{\bf d}+{\bf 1}}\bar{R}.$$
The following claim shows that, the map $\Phi\colon K[[x, y]] \longrightarrow  K[[x,y]]$ sending $x,y$ to $ x-g_1(x,y), y-g_2(x,y)$ respectively, is an automorphism of $K[[x, y]]$ and hence $\psi \sim_{\mathcal A} \psi'$ as required, since $\psi\circ \Phi=\psi'$.
\begin{claim}
$\mathrm{mt}(g_1)>1$ (similarly, $\mathrm{mt}(g_2)>1$).
\end{claim}
{\em Proof of the claim}: Since the case $\mathrm{mt}(f)=1$ is evident, we assume that $\mathrm{mt}(f)\geq 2$. We argue by contradiction. Suppose that it is not true, i.e. $\mathrm{mt}(g_1)=1$. Then by the definition of the maximal contact multiplicity $\bar{\beta_1}(f)$,
\begin{equation}\label{eq22}
\min \{i(f_i,g_1)|i=1,\ldots,r\}\leq \bar{\beta_1}(f).
\end{equation}  
The following three steps comprise the proof:\\
{\bf Step 1:} $\mathrm{mt}(f)=2$ and $r=1$. Then $d=c+1$ and $\psi(g_1)\in t^{d+1} K[[t]]$. This implies
$$i(f,g_1)=\text{ord }\psi(g_1)\geq d+1=c+2=\bar{\beta_1}(f)+1,$$
where the last equality is due to Lemma \ref{lm22}. This contradicts to (\ref{eq22}).\\
{\bf Step 2:} $\mathrm{mt}(f)=2$ and $r=2$. Then $f=f_1\cdot f_2$ with $\mathrm{mt}(f_1)=\mathrm{mt}(f_2)=1$ and ${\bf d}={\bf c}$. It follows from (\ref{eq21}) that $c_1=c_2=i(f_1,f_2)$. Since $\psi_1(g_1)\in t^{d_1+1}K[[t]]$, 
$$i(g_1,f_1)=\text{ord }\psi_1(g_1)\geq d_1+1=i(f_1,f_2)+1.$$
Similarly, $i(g_1,f_2)\geq i(f_1,f_2)+1.$ Combining Lemma \ref{lm21} and (\ref{eq22}) we get
$$i(f_1,f_2)+1\leq\min\{i(f_1,g_1);i(f_2,g_1)\}\leq \bar{\beta_1}(f)=i(f_1,f_2),$$
a contradiction.\\
{\bf Step 3:} $\mathrm{mt}(f)>2$. Then ${\bf d}={\bf c}-{\bf 1}$. Let $f=f_1\cdot \ldots\cdot f_r$ be an irreducible decomposition of $f$ such that $\mathrm{mt}(f_1)\leq \ldots \leq \mathrm{mt}(f_r)$. We consider the three following cases:

$\bullet$ If $\mathrm{mt}(f_r)>2$, then $i(f_r,g_1)=\text{ord }\psi_r(g_1)\geq d_r+1=c_r$. By Lemma \ref{lm22} and by the definition of the maximal contact multiplicity of $f_r$, one deduce that
$$c(f_r)>\bar{\beta_1}(f_r) \geq i(f_r,g_1)\geq c_r>c(f_r),$$
a contradiction.

$\bullet$ If $\mathrm{mt}(f_r)=2$, then $r>1$ and $i(f_r,g_1)=\text{ord }\psi_r(g_1)\geq d_r+1=c_r$. This implies that $\bar{\beta_1}(f_r) \geq c_r$. By (\ref{eq21}) and the inequality $i(f_1,f_r)\geq \mathrm{mt}(f_r)=2$, 
$$c_r\geq c(f_r)+i(f_1,f_r)>c(f_r)+1.$$
It follows from Lemma \ref{lm22} that $c(f_r)=\bar{\beta_1}(f_r)-1 \geq c_r-1>c(f_r),$ which is a contradiction.

$\bullet$ If $\mathrm{mt}(f_r)=1$ then $\mathrm{mt}(f_1)=\mathrm{mt}(f_2)=\ldots=\mathrm{mt}(f_r)=1$ and $r=\mathrm{mt}(f)>2$. Due to (\ref{eq21}) one has $c_1\geq i(f_1,f_2)+i(f_1,f_r) \geq i(f_1,f_2)+1$.  Hence
$$i(f_1,g_1)=\text{ord }\psi_1(g_1)\geq d_1+1=c_1\geq i(f_1,f_2)+1.$$
Similarly $i(f_2,g_1)\geq i(f_1,f_2)+1$ and then $i(f_1,f_2)+1\leq \min\{i(f_1,g_1);i(f_2,g_1)\}$. It hence follows from Lemma \ref{lm21} that
$$i(f_1,f_2)+1\leq \min\{i(f_1,g_1);i(f_2,g_1)\}\leq \bar{\beta_1}(f_1\cdot f_2)=i(f_1, f_2),$$
a contradiction. This completes the theorem.
\end{proof}
\begin{example}\label{ex21}{\rm
1. Let $f=x^2-y^5$. Then $r(f)=1$ and ${\bf c}(f)=4$. It is easy to see that $f$ is not parametrization $4$-determined.

2. Let $f=(x-y^3)(x-y^5)$. Then $r(f)=2$, ${\bf c}(f)=(3,3)$ and 
$$\psi\colon K[[x,y]]\to K[[t]]\oplus K[[t]], g\mapsto g(t^3,t)\oplus g(t^5,t)$$
is a parametrization of $f$. It can be easily verified that $f$ is parametrization $(3,2)$- but not $(2,2)$- determined.
}\end{example}
\section{Parametrization simple singularities}\label{sec3}
\subsection{Parametrization modality}
Consider an action of algebraic group $G$ on a variety $X$ (over a given algebraically closed field $K$) and a {\em Rosenlicht stratification} $\{(X_i,p_i), i=1,\ldots, s\}$ of $X$ w.r.t. $G$. That is, a stratification $X=\cup_{i=1}^sX_i$, where the stratum $X_i$ is a locally closed $G$-invariant subvariety of $X$ such that the projection $p_i:X_i\to X_i/G$ is a geometric quotient. For each open subset $U\subset X$ the modality of $U$, $G\text{-}\mathrm{mod}(U)$, is the maximal dimension of the images of $U \cap X_i$ in $X_i/G$. The modality $G\text{-}\mathrm{mod}(x)$ of a point $x \in X$ is the minimum of $G\text{-}\mathrm{mod}(U)$ over all open neighbourhoods $U$ of $x$.

Let $\mathcal L:=Aut(K[[x,y]])$ resp. $\mathcal R:=Aut(\bar{R})$ the left group resp. the right group. The left-right group $\mathcal A:=\mathcal R\times \mathcal L$ acts on $J=\mathrm{Hom}_K(K[[x,y]],\bar{R})$ by $\left((\phi,\Phi),\psi\right) \mapsto \Phi^{-1}\circ \psi\circ \phi$. Then, two elements $\psi,\psi'\in J$ are left-right equivalent, if and only if they belong to the same $\mathcal{A}$-orbit.

For each ${k}\in \mathbb{Z}$, denoted by $J_{k}$ the ${k}$-jet space of $J$, that is, the space of morphisms $$K[[x,y]]\to \bar{R}_{k}:=\bigoplus_{i=1}^r K[[t]]/(t^{k+1}).$$ We may identify an element $\psi$ in $J_{k}$ with the pair $(\psi(x),\psi(y))$ in $K[[t]]/(t^{k+1})\times K[[t]]/(t^{k+1})$, and therefore $J_k$ can be identified with the variety $\bar{R}_{k}^2\cong\mathbb{A}^{2(k+1)}_K$. For each element $\psi\in J$, denoted the $j^{k}\psi$ the image of $\psi$ by the map induced by the projection $\bar{R}\to \bar{R}_{k}$. We call $\psi$ to be {\em left-right $k$-determined} if it is left-right equivalent to any element in $J$ whose $k$-jet coincides with $j^{k}\psi$. A number $k$ is called {\em left-right sufficiently large} for $\psi$, if there exists a neighbourhood $U$ of the $j^k \psi$ in $J_k$ such that every $\psi'\in J$ with $j^k\psi'\in U$ is left-right $k$-determined. We also consider the ${k}$-jet of the left-right group $\mathcal A$ defined by $\mathcal A_{k}:=\mathcal R_{k}\times \mathcal L_{k}$. This group acts naturally on the ${k}$-jet space $J_{k}$. The {\em left-right modality} of $\psi$, $\mathcal A\text{-}\mathrm{mod}(\psi)$, is defined to be the $\mathcal A_k\text{-}\mathrm{modality}$ of $j^k \psi$ in $J_k$ with $k$ right sufficiently large for $\psi$. 

Let $f\in\mathfrak{m}\subset K[[x,y]]$ be reduced plane curve singularity and let $\psi$ be its parametrization. By Theorem \ref{thm21}, $\psi$ is left-right $(|{\bf c}|+1)$-determined, where $|{\bf c}|$ denotes the sum $c_1+\ldots+c_r$ for ${\bf c}=(c_1,\ldots,c_r)$. Note that, 
$$|{\bf c}|=\sum_{i=1}^r \left(c(f_i)+\sum_{j\neq i} i(f_i,f_j)\right)=2\delta(f).$$
It yields that  $\psi$ is left-right $(2\delta(f)+1)$-determined. From the upper semi-continuity of the delta function $\delta$ (see \cite{CL06}), we can show, by using the same argument as in \cite{GN16}, that $k=2\delta(f)+1$ is left-right sufficiently large for $\psi$. The {\em parametrization modality} of $f$, denoted by $\mathcal P\text{-}\mathrm{mod}(f)$, is defined to be the left-right modality of $\psi$, i.e the number $\mathcal A_k\text{-}\mathrm{mod}(j^k \psi)$.

A plane curve singularity $f\in K[[x,y]]$ is called {\em parametrization} {\em simple, uni-modal, bi-modal} or {\em $r$-modal} if its parametrization modality is equal to 0,1,2 or $r$ respectively. These notions are independent of the choice of a parametrization, and its sufficiently large number $k$. This may be proved in much the same way as \cite[Prop. 2.6, 2.12]{GN16}. The simpleness can be also described by deformation theory. A plane curve singularity $f\in K[[x,y]]$ is parametrization simple if its parametrization is of finite deformation type, i.e. its parametrization can be deformed only into finitely many left-right classes in $J$. 
\subsection{Parametrization simple irreducible plane curve singularities}
\begin{theorem}\label{thm31}
Let $p=\mathrm{char}(K)$. An irreducible plane curve singularity $f\in \mathfrak{m}^2\subset K[[x,y]]$ is parametrization simple if and only if one of its parametrizations is left-right equivalent to one of the singularities in the Tables \ref{table1}, \ref{table2}, \ref{table3} (where $\varepsilon\in \{0,1\}$ and  $c_k(y)=a_0+a_1y+\ldots+a_{k}y^{k}\in K[y]$).
\end{theorem}
\begin{proof}
The theorem follows from Propositions \ref{prop32} and \ref{prop35} below. 
\end{proof}

\begin{table}[h]
\begin{tabular}[20pt]{cl l l}
\hline
Name& Equations&Parametrizations&Conditions\\
\hline
$\mathrm{A}_{2k}$&$x^2+y^{2k+1}$&$(t^2,t^{2k+1})$&$k\geq 1$\\

$\mathrm{E}_{6}$&$x^3+y^{4}$& $(t^3,t^{4})$& \\

$\mathrm{E}_{8}$&$x^3+y^{5}$& $(t^3,t^{5})$& $p>5$\\

&$x^3+y^{5}+\varepsilon xy^4$& $(t^3,t^{5}+\varepsilon t^{7})$& $p=5$\\

$\mathrm{E}_{6k}$&$x^3+y^{3k+1}+c_{k-2}(y) x^2y^{2k+1}$& $(t^3,t^{3k+1}+\varepsilon t^{3(k+q)+2})$&$ 0\leq q\leq k-1$\\
&& &$q< k-1$ if $p\nmid 3k+1$\\

$\mathrm{E}_{6k+2}$&$x^3+y^{3k+2}+c_{k-2}(y) x^2y^{2k+1}$& $(t^3,t^{3k+2}+\varepsilon t^{3(k+q)+4})$&$ 0\leq q\leq k-1$\\
&& &$q< k-1$ if $p\nmid 3k+2$\\

$\mathrm{W}_{12}$&$x^4+y^5+a x^2y^3$& $(t^4,t^{5}+\varepsilon t^{q})$&$q=6,7,11$ \\
&& &$q\neq 6$ if $p>5$\\
$\mathrm{W}_{18}$&$x^4+y^7+c_1(y) x^2y^4$& $(t^4,t^{7}+\varepsilon t^{q})$&$p\neq 7;\ q=9,13$\\

$\mathrm{W}^{\sharp}_{2q-1}$&$(x^2+y^3)^2+c_1(y) xy^{q+4}$& $(t^4,t^{6}+ t^{2q+5})$&$\ q\geq 1$\\

\hline
\end{tabular}
\newline

\caption{Irreducible simple plane curve singularities ($p>3$).
}\label{table1}
\end{table}
\begin{table}[h]
\begin{tabular}[20pt]{cl l l}
\hline
Name& Equations&Parametrizations&Conditions\\
\hline
$\mathrm{A}_{2k}$&$x^2+y^{2k+1}$&$(t^2,t^{2k+1})$&$1\leq k$\\

$\mathrm{E}_{6}$&$x^3+y^{4}+\varepsilon x^2y^2$& $(t^3+\varepsilon t^{5},t^{4})$& \\

$\mathrm{E}_{8}$&$x^3+y^{5}$& $(t^3+\varepsilon t^{4}),t^{5}$& \\

$\mathrm{W}_{12}$&$x^4+y^5+a x^2y^3$& $(t^4,t^{5}+\varepsilon t^{q})$&$q=7,11$\\

\hline
\end{tabular}
\newline

\caption{Irreducible simple plane curve singularities in characteristic $3$. 
}\label{table2}
\end{table}

\begin{table}[h]
\begin{tabular}[20pt]{cl l l}
\hline
Name& Equations&Parametrizations&Conditions\\
\hline
$\mathrm{A}_{2k}$&$x^2+y^{2k+1}+\varepsilon xy^{2k-q}$&$(t^{2k+1},t^2+\varepsilon t^{2q+1})$&$1\leq q<k$\\

$\mathrm{E}_{6}$&$x^3+y^{4}+\varepsilon x^2y^2$& $(t^{4}+\varepsilon t^{5},t^3)$& \\

$\mathrm{E}_{8}$&$x^3+y^{5}$& $(t^{5},t^3)$& \\

$\mathrm{E}_{12}$&$x^3+y^{7}+\varepsilon x^2y^{5}$& $(t^3,t^{7}+\varepsilon t^{8})$&\\

\hline
\end{tabular}
\newline

\caption{Irreducible simple plane curve singularities in characteristic $2$. 
}\label{table3}
\end{table}

\begin{proposition}\label{prop32}
Let $f\in \mathfrak{m}^2\subset K[[x,y]]$ be an irreducible plane curve singularity and let $\left(x(t),y(t)\right)$ be its parametrization with $m=\mathrm{ord} x(t)=\mathrm{mt}(f)<\mathrm{ord} y(t)=n$. Then $f$ is not parametrization simple if either
\begin{itemize}
\item[(i)] $m>4$ or $\mathrm{(ii)}$ $m=4$ and $p=2$ or
\item[(iii)] $m=4$ and $n>7$ or $\mathrm{(iv)}$ $m=4$ and $n=7$ and $p=7$, or 
\item[(v)]  $m\geq 3, n\geq 6$ and $p =3$ or $\mathrm{(vi)}$ $m=3$ and $n\geq 8$ and $p=2$.
\end{itemize}
\end{proposition}
For the proof of these theorems we need the following lemma which is deduced from Corollaries A.4, A.9, A.10 of \cite{GN16} (see \cite[Prop. 3.2.4, Cor. 3.3.4 and Cor. 3.3.6]{Ng13} for more details).
\begin{lemma}\label{lm31}
Let the algebraic groups $G$ resp. $G'$ act on the varieties $X$ resp. $X'$. Let 
$h: Y\to X$ a morphism  of varieties and let $h':Y\to X'$ an open morphism such that 
$$
h^{-1}(G\cdot h(y))\subset h'^{-1}(G'\cdot h' (y)), \forall y\in Y.
$$
Then for all $y\in Y$ we have
$$G\text{-}\mathrm{mod}(h(y))\geq G'\text{-}\mathrm{mod}(h'(y))\geq \dim X'-\dim G'.$$
\end{lemma}
\begin{proof}[Proof of Proposition \ref{prop32}] We give a proof for (iv), since the others are similar and simpler. Let $k$ be left-right sufficiently large for the paramterization $(x(t),y(t))$ of $f$. Let $X:=J_k$, $G:=\mathcal{A}_k$.
We denote $Y:=\{\psi\in J_{k}\mid \psi (x)= at^4,a\neq 0,\mathrm{ord}\ \psi (y)\geq 7\}$,
$$X':=\{\psi\in J_{10}\mid \psi (x)= at^4,a\neq 0,\mathrm{ord}\ \psi (y)\geq 7\}\cong (\mathbb{A}\setminus \{0\})\times\mathbb{A}^{4},$$
 $G':=\mathcal{R}_1\times\mathcal{L}'_2$ with $\dim G'=4$ and define $h\colon Y\to X$ and $h'\colon Y\to X'$ to be the natural inclusion and projection respectively, where 
$$\mathcal{L}'_2:=\left\{\left(a_{10}x+\mathfrak{m}(3),b_{01}y+b_{20}x^2+\mathfrak{m}(3)\right)\in \mathcal{L}_2\right\}.$$
Here and below, for each $k\geq 1$, $\mathfrak{m}(k)$ resp. $O(k)$ stands for a series of multiplicity resp. of order at least $k$. We are going to show that 
$$
h^{-1}(G\cdot h(y))\subset h'^{-1}(G'\cdot h' (y)), \forall y\in Y.
$$
Indeed, for a given $\psi=(at^4,b_1t^7+b_2t^{8}+b_3t^{9}+b_4t^{10}+O(11))\in Y$, assume that $\psi'=(a't^4,b'_1t^7+b'_2t^{8}+b'_3t^{9}+b'_4t^{10}+O(11))\in Y$ such that $h(\psi')\in G.h(\psi)$. That is, there exist $\phi=c_1t+c_2t^2+\ldots \in \mathcal{R}_k$,
$$\Phi=(a_{10}x+a_{01}y+\mathfrak{m}(2), b_{10}x+b_{01}y+\mathfrak{m}(2))\in \mathcal{L}_k$$
such that $\psi'=\Phi\circ \psi\circ \phi$. This implies that $c_2=c_3=0$ and therefore
$$a'_1=a_{10}c^4_1a_1, b'_i=b_{01}c^7_1 b_1, b'_2=b_{01}c^{10}_1 b_2+b_{20}a^2_1, b'_3=b_{01}c^{9}_1 b_3, b'_4=b_{01}c^{10}_1 b_4.$$
Putting $\phi':=c_1t\in \mathcal{R}_1$ and $\Phi=(a_{10}x, b_{01}y+b_{20} x^2)\in \mathcal{L}'_2$
we get that $\psi'=\Phi'\circ \psi\circ \phi'$, i.e. $h'(\psi')\in G'.h'(\psi)$. It follows from Lemma \ref{lm31} that
 $$\mathcal{A}\text{-}\mathrm{mod}(\psi)\geq \dim X'-\dim G'=1.$$

\end{proof}
\begin{proposition}\label{prop35}
Let $f\in \mathfrak{m}^2\subset K[[x,y]]$ be an irreducible plane curve singularity and let $\left(x(t),y(t)\right)$ be its parametrization with $m=\mathrm{ord} x(t)<\mathrm{ord} y(t)=n$. 
\begin{itemize}
\item[(i)] If $m=2$ and $c(f)=2k$ then $f$ is parametrization equivalent to a singularity of type $A_k$.
\item[(ii)] If $m=3$ and $c(f)=q$ ($q=6,8$ if $p=3$; $q=6,8,12$ if $p=2$), then $f$ is parametrization equivalent to a singularity of type $E_q$.
\item[(iii)] If $m=4, n\neq 6$ and $c(f)=q$, then $f$ is parametrization equivalent to a singularity of type $W_q$.
\item[(iv)] If $m=4, n= 6$ and $c(f)=2q+14$, then $f$ is parametrization equivalent to a singularity of type $W^{\sharp}_{2q-1}$.
\end{itemize}
\end{proposition}

\begin{proof}
We prove only for the case $W^{\sharp}_{1}$, the other cases are proved completely similar. Assume that $m=4,n=6$ and $c(f)=16$. Since $m=4$ is not divisible by $p$, we may assume that $x(t)=t^4$. Let $k$ be the smallest odd exponent with non zero coefficient in $y(t)$. By \cite[Prop. 2.3.9]{Ng13}, $$16=c(f)=2\delta(f)=(6-1)(4-2)+(k-1)(2-1),$$
and hence  $k=7$. It is easy to see that  $\left(x(t),y(t)\right)$ is left-right equivalent to $\left(x_0(t),y_0(t)\right)$ of form
$$\left(t^4,t^6+t^{7}+O(8)\right).$$
We shall show that, $\left(x_0(t),y_0(t)\right)$ is left-right equivalent to a singularity of type $W^{\sharp}_{1}$. By Theorem \ref{thm21}, it suffices to prove that there exist $\phi\in \mathcal{R}$ and $\Phi\in \mathcal{L}$ such that
$$\Phi\left(x_0(\phi(t)),y_0(\phi(t))\right)= \left(t^4+0(16),t^6+t^{7}+0(16)\right).$$ 
We construct a sequence of equivalent elements $\left(x_i(t),y_i(t)\right), i=1,\ldots,5$ by constructing discrete automorphisms $\phi_i\in \mathcal{R}$ and automorphisms $\Phi_i\in \mathcal{L}, i=1,\ldots,5$ such that 
$$\left(x_{i}(t),y_{i}(t)\right)=\Phi_i\left(x_{i-1}(\phi_i(t)),y_{i-1}(\phi_i(t))\right), \forall i=1,\ldots,5$$
and 
$$\left(x_5(t),y_5(t)\right)= \left(t^4+0(16),t^6+t^{7}+0(16)\right).$$
Indeed, we may write
$$\left(x_0(t),y_0(t)\right)=\left(t^4,t^6+t^{7}+b_{8}t^8+b_{9}t^9+0(10)\right),$$
for some $b_{8},b_{9}\in K$, and define
$$\phi_1(t)=t+ct^3+ct^{4},\ \Phi_1(x,y)=\left(x-4cy,y\right)$$
with $c=(b_9-b_8)/7$. Then 
$$\left(x_1(t),y_1(t)\right)=\left(t^4+O(8),t^6+t^{7}+O(10)\right).$$
Since $4$ is not divisible by $p$, there exists an automorphism $\phi_2\in \mathcal{R}$ such that\\ $x_1(\phi_2(t))=t^4$. Putting $\Phi_2(x,y)=\left(x,y\right)$ one has
$$\left(x_2(t),y_2(t)\right)=\left(t^4,t^6+t^{7}+b_{10}t^{10}+b_{11}t^{11}+O(12)\right),$$
for some $b_{10},b_{11}\in K$. We define
$$\phi_3(t)=t+(b_{10}-b_{11})t^5,\ \Phi_3(x,y)=\left(x,y-(7b_{10}-6b_{11})xy\right).$$
Then
$$\left(x_3(t),y_3(t)\right)=\left(t^4+a_{8}t^8+a_{12}t^{12}+0(16),\ t^6+t^{7}+b_{12}t^{12}+b_{13}t^{13}+O(14)\right),$$
for some $a_{8},a_{12}, b_{12},b_{13}\in K$. The automorphism $\phi_4(t)=t$ and the automorphism $$\Phi_4(x,y)=\left(x-a_8x^2+(2a_8-a_{12})x^3,y-\frac{b_{13}}{2} y^2+(\frac{b_{13}}{2}-b_{12})x^3\right)$$ yield that
$$\left(x_4(t),y_4(t)\right)=\left(t^4+0(16),\ t^6+t^{7}+b_{14}t^{14}+b_{15}t^{15}+O(16)\right).$$
Applying  
$$\phi_5(t)=t+(b_{14}-b_{15})t^9,\ \Phi_5(x,y)=\left(x-4(b_{14}-b_{15})x^3,y-(7b_{14}-6b_{15})x^2y\right)$$
to $\left(x_4(t),y_4(t)\right)$ we obtain
$$\left(x_5(t),y_5(t)\right)=\left(t^4+0(16),\ t^6+t^{7}+O(16)\right)$$
as desired.
\end{proof}

\begin{proposition}
The singularities in Tables \ref{table3} are parametrization simple.
\end{proposition}

\begin{proof}
It follows directly from Proposition \ref{prop35} and the upper semicontinuity of the multiplicity $m$ and of the conductor $c$. For instance, a singularity of type $W^{\sharp}_{2q-1}$ ($p>3$) can be deformed into at most the classes $A_{k}, E_{k}, W_{12}, W_{18}, W^{\sharp}_{2q-1}$ with $k\leq 2q+14$.
\end{proof}

\end{document}